\documentclass[smallextended]{svjour3}       
\smartqed  
\usepackage{graphicx}
\usepackage{amsmath,amsfonts,amssymb}
\usepackage{enumerate}
\usepackage[ruled,lined,boxed]{algorithm2e}
\usepackage{xcolor}
\usepackage{subcaption}
\DeclareMathOperator*{\minimize}{minimize}
\SetKwBlock{Repeat}{repeat}{end}

\def\trt{^{\scriptscriptstyle T}}

\begin{document}

\title{On the solution of monotone nested variational inequalities}


\author{Lorenzo Lampariello \and Gianluca Priori \and Simone Sagratella
}


\institute{L. Lampariello \at
              Department of Business Studies, Roma Tre University, Italy \\
         \email{lorenzo.lampariello@uniroma3.it}           
    \and
    G. Priori and S. Sagratella \at
    		  Department of Computer, Control and Management Engineering Antonio Ruberti, Sapienza University of Rome, Italy \\
         \email{priori,sagratella@diag.uniroma1.it}           
}

\date{Received: date / Accepted: date}

\maketitle

\begin{abstract}
We study nested variational inequalities, which are variational inequalities whose feasible set is the solution set of another variational inequality.
We present a projected averaging Tikhonov algorithm requiring the weakest conditions in the literature to guarantee the convergence to solutions of the nested variational inequality.
Specifically, we only need monotonicity of the upper- and the lower-level variational inequalities.
Also, we provide the first complexity analysis for nested variational inequalities
considering optimality of both the upper- and lower-level.

\end{abstract}

\section{Introduction}
We focus on solving (upper-level) variational inequalities whose feasible set is given by the solution set of another (lower-level) variational inequality.
These problems are commonly referred to as nested variational inequalities \textcolor{black}{and they represent a flexible modelling tool when it comes to solving, for instance, problems arising in finance (see, e.g., the most recent \cite{lampariello2021equilibrium} for an application in the multi-portfolio selection context) as well as several other fields (see, e.g., \cite{facchinei2014vi,scutari2012equilibrium} for resource allocation problems in communications and networking). }  

\textcolor{black}{As far as the literature on nested VIs is concerned, it is still in its infancy if compared to the bilevel instance of hierarchical optimization or, more generally, bilevel structures as in \cite{dempe2002foundations,lampariello2017bridge,lampariello2020numerically,lampariello2019standard}. However, there are two main solution methods that are most adopted in the field-related literature: hybrid-like techniques (see, e.g., \cite{lu2009hybrid,marino2011explicit,yamada2001hybrid}) and Tikhonov-type schemes (see, e.g. \cite{lampariello2020explicit} for the latest developments and \cite{kalashnikov1996solving} for some earlier developments, as well as \cite{facchinei2014vi} and the references therein). It should be pointed out that hybrid-like procedures very often require particularly strong assumptions (e.g. demanding co-coercivity of the lower-level map as in \cite{lu2009hybrid,marino2011explicit,yamada2001hybrid}) in order to ensure convergence or to work properly at all. Hence, we rely on the Tikhonov paradigm drawing from the general schemes proposed in \cite{facchinei2014vi} and \cite{lampariello2020explicit}, however asking for less stringent assumptions and combining the Tikhonov approach with a new averaging procedure.}   

We widen the scope and expand the applicability of nested variational inequalities by showing for the first time in related literature that solutions can be provably computed in the more general framework of simply monotone upper- and lower-level variational inequalities. Specifically, in \cite{facchinei2014vi,lampariello2020explicit}, where as far as we are aware the most advanced results are obtained, the upper-level map is required to be monotone plus. Relying on a combination of a Tikhonov approach with an averaging procedure, the algorithm we propose is shown to converge provably to a solution of a nested variational inequality where the upper-level map is required to be just monotone. 

We also obtain complexity results for our method. Except for \cite{lampariello2020explicit}, not only does this analysis represent the only other complexity study in the literature of nested variational inequalities, it is also the first one in the field dealing with upper-level optimality.

\section{Problem definition and motivation}\label{sec:mot}
Let us consider the nested variational inequality $\text{VI}\big(G, \text{SOL}(F,Y)\big)$,
where $G: \mathbb R^{n} \to \mathbb R^n$ is the upper-level map, and $\text{SOL}(F,Y)$ is the solution set of the lower-level VI$(F,Y)$. \textcolor{black}{We recall that, given a subset $Y \subseteq \mathbb R^n$ and a mapping $F: \mathbb R^{n} \to \mathbb R^n$, the variational inequality VI$(F,Y)$ is the problem of computing a vector $x\in \mathbb R^n$ such that
\begin{equation}\label{eq:VIbasic}
x \in Y, \quad F(x)\trt (y-x) \ge 0,\quad\forall\,y\in Y.
\end{equation}  }
In other words, $\text{VI}\big(G, \text{SOL}(F,Y)\big)$ is the problem of finding $x \in \mathbb R^n$ that solves
\begin{equation}\label{eq:pureh}
x \in \text{SOL}(F,Y), \quad G(x)\trt (y-x) \ge 0,\quad\forall\,y\in \text{SOL}(F,Y).
\end{equation} 
\textcolor{black}{As it is clear from \eqref{eq:pureh}, the feasible set of $\text{VI}\big(G, \text{SOL}(F,Y)\big)$ is implicitly defined as the solution set of the lower-level VI$(F,Y)$.}
The nested variational inequality \eqref{eq:pureh} we consider has a purely hierarchical structure in that the lower-level problem \eqref{eq:VIbasic} is non parametric with respect to the upper-level variables, \textcolor{black}{unlike the more general bilevel structures presented in \cite{lampariello2020numerically}.
Under mild conditions, VIs equivalently reformulate NEPs so that in turn, by means of structure \eqref{eq:pureh}, we are able to model well-known instances of bilevel optimization and address multi-follower games. }

We introduce the following blanket assumptions which are widely adopted in the literature of solution methods for variational inequalities:
\begin{enumerate}
\item[(A1)]
the upper-level map $G$ is Lipschitz continuous with constant $L_{G}$ and monotone on $Y$;   
\item[(A2)]
the lower-level map $F$ is Lipschitz continuous with constant $L_F$ and monotone on $Y$;
\item[(A3)]
 the lower-level feasible set $Y$ is nonempty, convex and compact.
\end{enumerate}

\noindent
Due to (A2) and (A3), SOL$(F,Y)$ is a nonempty, convex, compact and not necessarily single-valued set, see e.g. \cite[Section 2.3]{FacchPangBk}. As a consequence, the feasible set of the nested variational inequality \eqref{eq:pureh} is not necessarily a singleton. Moreover, thanks to (A1), the solution set of the nested variational inequality \eqref{eq:pureh} can include multiple points. 

Notice that assumption (A1) on the upper-level map $G$ is much less demanding than the one required in \cite{facchinei2014vi,lampariello2020explicit}. Specifically, here we assume $G$ to be only monotone, while in \cite{facchinei2014vi,lampariello2020explicit} it must be monotone plus.

\textcolor{black}{For the sake of completeness, we recall that a mapping $G: Y \subseteq \mathbb R^{n} \to \mathbb R^n$ is said to be monotone \textit{plus} on $Y$ if both the following conditions hold:
\begin{enumerate}
    \item G is monotone on $Y$, i.e. $(G(x)-G(y))\trt (x-y) \geq 0 \quad \forall \; x,y \in Y$; 
    \item $(x-y)\trt (G(x)-G(y))=0 \Rightarrow G(x)=G(y) \quad \forall \; x,y \in Y$.
\end{enumerate}
Consequently, we are now able to dispose of the monotonicity \textit{plus} assumption on operator $G$, which inevitably represents a more stringent condition when compared to plain monotonicity, thus hereby asking for $G$ to be simply monotone. Indeed, whenever the upper-level map G is nonsymmetric, requiring G to be monotone \textit{plus} is ``slightly less'' than assuming G to be strongly monotone (see, e.g. \cite{bigi2021combining}).}    
In fact, the main objective of this work is to define, for the first time in the field-related literature, an algorithm (Algorithm \ref{al:algo}) which is able to compute solutions of the monotone nested variational inequality \eqref{eq:pureh}, under the weaker assumptions (A1)-(A3), see the forthcoming Theorem \ref{th:conv2}.

In addition, we study complexity properties of Algorithm \ref{al:algo} in detail, see Theorems \ref{th:ftermination} and \ref{th:ftermination easy}.
We highlight that all steps in Algorithm \ref{al:algo} can be readily implemented and no nontrivial computations are required, see e.g. the numerical illustration in section \ref{sec: numerical}.

Summarizing,
\begin{itemize}
 \item we show that the algorithm we propose is globally (subsequential) convergent to solutions of monotone nested variational inequalities under the weakest conditions in the literature,
 \item we provide the first complexity analysis for nested variational inequalities considering optimality of both the upper- and lower-level (instead in \cite{lampariello2020explicit} just lower-level optimality is contemplated).
\end{itemize}

\section{A projected averaging Tikhonov algorithm}\label{sec:algo}

For the sake of notation, let us introduce the following operator:
\[
\Phi_\tau(x) \triangleq F(x) + \frac{1}{\tau} G(x),
\]
which is the classical operator used to define subproblems in Tikhonov-like methods.
For any $\tau \in \mathbb R_{++}$, by assumptions (A1) and (A2), $\Phi_\tau$ is monotone and Lipschitz continuous with constant $L_\Phi \triangleq L_{F} + L_{G}$ on $Y$.
Moreover, the following finite quantities are useful in the forthcoming analysis:
$$
H \triangleq \max_{y \in Y} \|G(y)\|_2, \quad R \triangleq \max_{y \in Y} \|F(y)\|_2, \quad D \triangleq \max_{v, y \in Y}\|v - y\|_2.
$$
\textcolor{black}{For the sake of clarity, let us recall that, by definition, the Euclidean projection $P_Y(x)$ of a vector $x \in \mathbb{R}^n$ onto a closed convex subset $Y \subseteq \mathbb{R}^n$ is the unique solution of the strongly convex (in $y$) problem
\begin{equation*}
\begin{aligned}
\minimize_{y} \quad & \frac{1}{2} \; (y-x)\trt (y-x)\\
\textrm{s.t.} \quad & y \in Y.\\
\end{aligned}    
\end{equation*}
The solution of the latter problem is a unique vector $\bar y \in Y$ that is closest to $x$ in the Euclidean norm (see, e.g. \cite[Th. 1.5.5]{FacchPangBk} for an exhaustive overview on the Euclidean projector and its properties).} 

\textcolor{black}{In our analysis we rely on approximate solutions of VIs. Specifically, we say that $x \in K$ approximately solves VI$(\Psi,K)$ (with $\Psi$ continuous and $K$ convex and compact) if 
\begin{equation}\label{eq: inexact optimality}
 \Psi (x)^\top(y-x) \geq -\varepsilon, \quad \forall \, y \in K,
\end{equation}
where $\varepsilon \geq 0$. Relation \eqref{eq: inexact optimality}, for example, when VI$(\Psi,K)$ defines the first-order optimality conditions of a convex problem, guarantees that the problem is solved up to accuracy $\varepsilon$. Relation \eqref{eq: inexact optimality}, in view of assumption (A3), is equivalent to
\begin{equation}\label{eq:stopping general}
\min_{y \in K} \; \Psi (x)^\top(y-x) \geq -\varepsilon.
\end{equation}
We remark that $\Psi (x)^\top(y-x)$ is linear in $y$; moreover, if $K$ is polyhedral (as, e.g., in the multi-portfolio selection context) computing $\min_{y \in K} \; \Psi (x)^\top(y-x)$ amounts to solving a linear optimization problem. In any event, we assume this computation to be easy to do in practice.}

With the following result we relate approximate solutions of the VI subproblem
\begin{equation}\label{eq: subproblem inexact optimality}
 \Phi_\tau (x)^\top(y-x) \geq -\varepsilon_{\text{sub}}, \quad \forall \, y \in Y,
\end{equation}
where $\varepsilon_{\text{sub}} \geq 0$, with approximate solutions of problem \eqref{eq:pureh}.

\begin{proposition}\label{th: upper level optimality}
Assume conditions (A1)-(A3) to hold, and let $x \in Y$ be a solution of the VI subproblem \eqref{eq: subproblem inexact optimality}
with $\tau > 0$ and $\varepsilon_{\text{sub}} \geq 0$.
It holds that
\begin{equation}\label{eq: problem inexact optimality}
G(x)^\top (y-x) \geq -\varepsilon_{\text{up}}, \quad \forall y \in \text{SOL}(F,Y),
\end{equation}
with $\varepsilon_{\text{up}} \geq \varepsilon_{\text{sub}} \tau$, and
\begin{equation}\label{eq: problem inexact feasibility}
F(x)^\top (y-x) \geq -\varepsilon_{\text{low}}, \quad \forall y \in Y,
\end{equation}
with $\varepsilon_{\text{low}} \geq \varepsilon_{\text{sub}} + \frac{1}{\tau} H D$.
\end{proposition}
\begin{proof}
 We have for all $y \in \text{SOL}(F,Y)$:
 \begin{align*}
 -\varepsilon_{\text{sub}} \tau & \leq \left[\tau F(x) + G(x)\right]^\top(y-x) \\
 & \leq \left[\tau F(y) + G(x)\right]^\top(y-x) \\
 & \leq G(x)^\top(y-x),
 \end{align*}
 where the first inequality is due to \eqref{eq: subproblem inexact optimality}, the second one comes from the monotonicity of $F$, and the last one is true because $x \in Y$ and then $F(y)^\top(x-y) \geq 0$. That is \eqref{eq: problem inexact optimality} is true.

 Moreover, we have for all $y \in Y$:
 \begin{align*}
 F(x)^\top (y-x) & = \Phi_\tau (x)^\top (y-x) - \frac{1}{\tau} G(x)^\top (y-x) \\
 & \geq -\varepsilon_{\text{sub}} - \frac{1}{\tau} H D,
 \end{align*}
 where the inequality is due to \eqref{eq: subproblem inexact optimality}. Therefore we get \eqref{eq: problem inexact feasibility}.
 \hfill $\square$
\end{proof}

\noindent
Proposition \ref{th: upper level optimality} suggests a way to solve, with a good degree of accuracy, the hierarchical problem \eqref{eq:pureh}. That is solving the VI subproblem \eqref{eq: subproblem inexact optimality} with a big value for $\tau$ and an $\varepsilon_{\text{sub}}$ sufficiently small in order to make $\varepsilon_{\text{sub}} \tau$ small enough.
Following this path, we propose a Projected Averaging Tikhonov Algorithm (PATA), see Algorithm \ref{al:algo}, to compute solutions of problem \eqref{eq:pureh}.

\medskip
{

\begin{algorithm}[H]
\KwData{$\{\bar \varepsilon^i\} \subseteq \mathbb R_+, \, \{\bar \tau^{i} \} \subseteq \mathbb R_{++}$, $\gamma^k \in (0,1]$, $w^0 = z^0 = y^0 \in Y$, $i, \, l \leftarrow 0$;}
\medskip
\For{$k=0,1,\ldots$}{
\nlset{(S.1)}{$\varepsilon^k = \bar \varepsilon^i$, $\tau^k = \bar \tau^{i}$;}

\nlset{(S.2)}{$y^{k+1} = P_Y(y^k - \gamma^{k-l} \Phi_{\tau^k}(y^k))$\label{S.1};}

\nlset{(S.3)}{$\displaystyle z^{k+1} = \frac{\sum_{j=l}^{k+1} \gamma^{j-l} y^j}{\sum_{j=l}^{k+1} \gamma^{j-l}}$\label{S.2};}

{\nlset{(S.4)} \If{$\min_{y \in Y} \Phi_{\tau^k}(z^{k+1})^\top (y - z^{k+1}) \geq -\varepsilon^k$} {$w^{i+1} = z^{k+1}$, $i=i+1$, $l=k+1$;} \label{S.3}}}{}

\caption{\label{al:algo} Projected Averaging Tikhonov Algorithm (PATA)}
\end{algorithm}}
\medskip
\noindent
Some comments about PATA are in order.
Index $i$ denotes the outer iterations that occur when the condition in step \ref{S.3} is verified, and they correspond to solutions $w^{i+1}$ of the VI subproblems \eqref{eq: subproblem inexact optimality} with $\varepsilon_{\text{sub}} = \bar \varepsilon^i$ and $\tau = \bar \tau^i$. The sequence $\{y^k\}$ is obtained by making classical projection steps with stepsizes $\gamma^k$, see step \ref{S.1}. The sequence $\{z^k\}$ consists of the inner iterations needed to compute a solution of the VI subproblem \eqref{eq: subproblem inexact optimality}, and it is obtained by performing a weighted average on the points $y^j$, see step \ref{S.2}.
Index $l$ is included in order to let the sequence of the stepsizes $\{\gamma^k\}$ restart at every outer iteration and to consider only the points $y^j$ belonging to the current subproblem to compute $z^{k+1}$.

We remark that the condition in step \ref{S.3} only requires the solution of an optimization problem with a linear objective function over the convex set $Y$ (see the discussion about inexact solutions of VIs below relation \eqref{eq:stopping general}). In section \ref{sec: numerical} we give a practical implementation of PATA.

In the following section we show that Proposition \ref{th: upper level optimality} can be used to prove that PATA effectively computes solutions of problem \eqref{eq:pureh}.

\section{Main convergence properties}\label{sec:conv}

First of all we deal with convergence properties of PATA.

\begin{theorem}\label{th:conv2}
Assume conditions (A1)-(A3) to hold, and let conditions
\begin{equation}\label{eq: conditions for convergence}
\sum_{k=0}^\infty \gamma^k = \infty, \enspace \frac{\sum_{k=0}^\infty (\gamma^k)^2}{\sum_{k=0}^\infty \gamma^k} = 0, \enspace
\frac{1}{\bar \tau^{i}} \downarrow 0, \enspace \sum_{i=0}^\infty \frac{1}{\bar \tau^{i}} = \infty, \enspace \bar \varepsilon^{i} = \frac{c}{(\bar \tau^{i})^\beta},
\end{equation}
hold with $\beta > 1$ and $c > 0$. Every limit point of the sequence $\{w^i\}$ generated by PATA is a solution of problem \eqref{eq:pureh}. 
\end{theorem}
\begin{proof}
 First of all we show that $i \to \infty$.
 Assume by contradiction that this is not true, therefore 
 there exists an index $\bar k$ such that the condition in step \ref{S.3} is violated for every $k \geq \bar k$, and either $\bar k = 0$ or the condition in step \ref{S.3} is satisfied at the iteration $\bar k - 1$. We denote $\bar \tau = \tau^{\bar k}$, and observe that $\tau^k = \bar \tau$ for every $k \ge \bar k$.

For every $j \in [\bar k, k]$, and for any $v \in Y$, we have
\[
\begin{array}{rcl}
	\|y^{j+1} - v\|_2^2 & = & \|P_Y(y^j - \gamma^{j-\bar k} \Phi_{\bar \tau}(y^j)) - v\|_2^2\\[5pt]
	& \le & \|y^j - \gamma^{j-\bar k} \Phi_{\bar \tau}(y^j) - v\|_2^2\\[5pt]
	& = & \|y^j - v\|_2^2 + (\gamma^{j-\bar k})^2 \|\Phi_{\bar \tau}(y^j)\|_2^2 - 2 \gamma^{j-\bar k} \Phi_{\bar \tau}(y^j)^\top (y^j - v),
\end{array}
\]
and, in turn,
\[
\Phi_{\bar \tau}(y^j)^\top (v - y^j) \ge \frac{\|y^{j+1} - v\|_2^2 - \|y^j - v\|_2^2}{2 \gamma^{j-\bar k}} - \frac{\gamma^{j-\bar k}}{2} \|\Phi_{\bar \tau}(y^j)\|_2^2. 
\]
Summing, we get
{\footnotesize
\begin{equation}\label{eq:twocases 2}
\begin{array}{rcl}
\frac{\displaystyle\sum_{j=\bar k}^{k} \gamma^{j-\bar k} \Phi_{\bar \tau}(y^j)^\top (v - y^j)}{\displaystyle\sum_{j=\bar k}^{k} \gamma^{j-\bar k}} & \ge & \frac{\displaystyle\sum_{j=\bar k}^{k} \left(\|y^{j+1} - v\|_2^2 - \|y^j - v\|_2^2 - (\gamma^{j-\bar k})^2 \|\Phi_{\bar \tau}(y^j)\|_2^2\right)}{2\displaystyle\sum_{j=\bar k}^{k} \gamma^{j-\bar k}} \\[5pt]
& = & \frac{\displaystyle\left(\|y^{k+1} - v\|_2^2 - \|y^{\bar k} - v\|_2^2 - \displaystyle\sum_{j=\bar k}^{k} (\gamma^{j-\bar k})^2 \|\Phi_{\bar \tau}(y^j)\|^2_2\right)}{2\displaystyle\sum_{j=\bar k}^{k} \gamma^{j-\bar k}} \\[5pt]
& \geq & - \frac{\displaystyle \left( \|y^{\bar k} - v\|_2^2 + \displaystyle\sum_{j=\bar k}^{k} (\gamma^{j-\bar k})^2 \|\Phi_{\bar \tau}(y^j)\|^2_2\right)}{2\displaystyle\sum_{j=\bar k}^{k} \gamma^{j-\bar k}},	
\end{array}
\end{equation}}
which implies  
\begin{equation}\label{eq:intermaver 2}
\begin{array}{rcl}
	 \Phi_{\bar \tau}(v)^\top (v - z^k) & = & \frac{1}{\displaystyle\sum_{j=\bar k}^{k} \gamma^{j-\bar k}} \displaystyle\sum_{j=\bar k}^{k} \gamma^{j-\bar k} \Phi_{\bar \tau}(v)^\top (v - y^j)\\[5pt]
	 & \ge & -\frac{\displaystyle\left(\|y^{\bar k} - v\|_2^2 + \sum_{j=\bar k}^{k} (\gamma^{j-\bar k})^2 \|\Phi_{\bar \tau}(y^j)\|^2_2\right)}{2\displaystyle\sum_{j=\bar k}^{k} \gamma^{j-\bar k}} \\[5pt]
	& & + \frac{1}{\displaystyle\sum_{j=\bar k}^{k} \gamma^{j-\bar k}} \displaystyle\sum_{j=\bar k}^{k} \gamma^{j-\bar k} (\Phi_{\bar \tau}(v) - \Phi_{\bar \tau}(y^j))^\top (v - y^j)\\[5pt]
	& \ge & -\frac{\displaystyle\left(\|y^{\bar k} - v\|_2^2 + \sum_{j=\bar k}^{k} (\gamma^{j-\bar k})^2 \|\Phi_{\bar \tau}(y^j)\|^2_2\right)}{2\displaystyle\sum_{j=\bar k}^{k} \gamma^{j-\bar k}},
\end{array}
\end{equation}
where the last inequality is due to the monotonicity of $\Phi_{\bar \tau}$. Denoting by $z \in Y$ any limit point of the sequence $\{z^k\}$, taking the limit in the latter relation and subsequencing, the following inequality holds:
\[
{\footnotesize{
\begin{array}{rcl}
	 \Phi_{\bar \tau}(v)^\top (v - z) & \ge & -\frac{\displaystyle\left(\|y^{\bar k} - v\|_2^2 + \sum_{j=\bar k}^{\infty} (\gamma^{j-\bar k})^2 \|\Phi_{\bar \tau}(y^j)\|^2_2\right)}{2\displaystyle\sum_{j=\bar k}^{\infty} \gamma^{j-\bar k}} = 0,
\end{array}}}
\]
because $\sum_{j=\bar k}^{\infty} \gamma^{j-\bar k} = +\infty$ and $\left(\sum_{j=\bar k}^{\infty} (\gamma^{j-\bar k})^2\right)/\left(\sum_{j=\bar k}^{\infty} \gamma^{j-\bar k}\right) = 0$, and then $z$ is a solution of the dual problem
\[
\Phi_{\bar \tau}(v)^\top (v - z) \ge 0, \enspace \forall v \in Y.
\] 
Hence, the sequence $\{z^k\}$ converges to a solution of VI$(Y, \Phi_{\bar \tau})$, see e.g. \cite[Theorem 2.3.5]{FacchPangBk}, in contradiction to $\min_{y \in Y} \Phi_{\bar \tau}(z^{k+1})^\top (y - z^{k+1}) < - \varepsilon^k = - \varepsilon^{\bar k}$ for every $k \ge \bar k$.

Therefore the algorithm produces an infinite sequence $\{w^i\}$ such that $w^{i+1} \in Y$ and
$$
 \Phi_{\bar \tau^i} (w^{i+1})^\top(y-w^{i+1}) \geq -\frac{c}{(\bar \tau^i)^\beta}, \quad \forall \, y \in Y,
$$
that is \eqref{eq: subproblem inexact optimality} holds at $w^{i+1}$ with $\varepsilon_{\text{sub}} = \frac{c}{(\bar \tau^i)^\beta}$.
By Proposition \ref{th: upper level optimality}, specifically from \eqref{eq: problem inexact optimality} and \eqref{eq: problem inexact feasibility}, we obtain
\begin{equation}\label{optimality G}
G(w^{i+1})^\top (y-w^{i+1}) \geq -\frac{c}{(\bar \tau^i)^{\beta-1}}, \quad \forall y \in \text{SOL}(F,Y),
\end{equation}
and
\begin{equation}\label{optimality F}
F(w^{i+1})^\top (y-w^{i+1}) \geq -\frac{c}{(\bar \tau^i)^{\beta}} - \frac{1}{\bar \tau^i} H D, \quad \forall y \in Y.
\end{equation}
Taking the limit $i \to \infty$, and recalling that $G$ and $F$ are continuous and $\beta > 1$, we get the desired convergence property.
\hfill $\square$

\end{proof}

\noindent 
Conditions \eqref{eq: conditions for convergence} for the sequence of stepsizes $\{\gamma^k\}$ are satisfied, e.g., if we choose
\begin{equation*}
\gamma^k = \min\left\{1,\frac{a}{k^\alpha}\right\},
\end{equation*}
with $a > 0$ and $\alpha \in (0,1]$, see Proposition \ref{th: sequence complies with theorem} in the Appendix. \textcolor{black}{Another possible choice of step-size rule satisfying conditions \eqref{eq: conditions for convergence}, as shown in \cite{facchinei2015parallel}, is
\begin{equation*}
    \gamma^k=\gamma^{k-1}(1-\theta \gamma^{k-1}), \quad k=1, \hdots
\end{equation*}
where $\theta \in (0,1)$ is a given constant, provided that $\gamma^0 \in (0,1]$. Both the above step-size rules satisfy conditions \eqref{eq: conditions for convergence}, needed for Theorem \ref{th:conv2} to be valid. }

\textcolor{black}{We remark that, even if we require assumptions that are less stringent with respect to related literature, we still obtain the same type of convergence as in related literature, namely \textit{subsequential} convergence to a solution of problem \eqref{eq:pureh}. Note that, thanks to assumption (A3), at least a limit point of sequence $\{w^i\}$ generated by PATA exists. As it is common practice when using an iterative algorithm like PATA, referring to \eqref{optimality G} and \eqref{optimality F}, $w^{i+1}$ can be considered an approximate solution of problem \eqref{eq:pureh} as soon as $\frac{c}{(\bar \tau^i)^{\beta-1}}$ and $\left(\frac{c}{(\bar \tau^i)^\beta} + \frac{1}{\bar \tau^i} H D \right)$ are small enough. Clearly, if the upper-level map $G$ is strongly monotone on $Y$, the whole sequence $\{w^i\}$ converges to the unique solution of problem \eqref{eq:pureh}.  }

\textcolor{black}{We consider the so-called natural residual map for VI$(\Psi, K)$ (with $\Psi$ continuous and $K$ convex and compact) 
\begin{equation}\label{eq:meritU}
U(x) \triangleq \|P_K(x - \Psi(x)) - x\|.
\end{equation}
As recalled in \cite{lampariello2020explicit}, the function $U$ is continuous and nonnegative. Moreover, $U(x) = 0$ if and only if $x \in \text{SOL}(\Psi, K)$. Specifically, classes of problems exist for which the value $U(x)$ also gives an actual upper-bound to the distance between $x$ and $\text{SOL}(\Psi, K)$, see \cite{lampariello2020explicit} and the references therein.
Therefore, the following condition
\begin{equation}\label{eq: problem inexact feasibility natural gen}
 U(x) \leq \widehat \varepsilon,
\end{equation}
with $\widehat \varepsilon \geq 0$, is alternative to \eqref{eq: inexact optimality}.
However, relations \eqref{eq: inexact optimality} and \eqref{eq: problem inexact feasibility natural gen} turn out to be related to each other: we show in Appendix (Proposition \ref{th: inexact feasibility relations}) that if $x$ satisfies \eqref{eq: inexact optimality}, then \eqref{eq: problem inexact feasibility natural gen} holds with $\widehat \varepsilon \geq \sqrt{\varepsilon}$. Vice versa, condition \eqref{eq: problem inexact feasibility natural gen} implies \eqref{eq: inexact optimality} with $\varepsilon \geq (\Omega + \Xi) \widehat \varepsilon$, where $\Omega \triangleq \max_{v, y \in K}\|v - y\|_2$ and $\Xi \triangleq \max_{y \in K} \|\Psi(y)\|_2$. }

In order to deal with the convergence rate analysis of our method, we consider the natural residual map for the lower-level VI$(F, Y)$ 
\begin{equation}\label{eq:meritV}
V(x) \triangleq \|P_Y(x - F(x)) - x\|.
\end{equation}
Clearly, the following condition
\begin{equation}\label{eq: problem inexact feasibility natural}
 V(x) \leq \widehat \varepsilon_{\text{low}},
\end{equation}
with $\widehat \varepsilon_{\text{low}} \geq 0$, is alternative to \eqref{eq: problem inexact feasibility} to take care of the feasibility of problem \eqref{eq:pureh}.

In this context, we underline that the convergence rate we establish is intended to give an upper bound to the number of iterations needed to drive both the upper-level error $\varepsilon_{\text{up}}$, given in \eqref{eq: problem inexact optimality}, and the lower-level error $\widehat \varepsilon_{\text{low}}$, given in \eqref{eq: problem inexact feasibility natural}, under some prescribed tolerances $\delta_{\text{up}}$ and $\widehat \delta_{\text{low}}$, respectively.

\begin{theorem}\label{th:ftermination}
Assume conditions (A1)-(A3) to hold and, without loss of generality, assume $L_\Phi < 1$. Consider PATA. Given some precisions $\delta_{\text{up}}, \widehat \delta_{\text{low}} \in (0,1)$, set $\gamma^k = \min\left\{1,\frac{1}{2 k^\frac{1}{2}}\right\}$, $\bar \tau^i = \max\{1,i\}$, and $\bar \varepsilon^i = \frac{1}{(\bar\tau^{i})^2}$.
Let us define the quantity
$$
 I_{\max} \triangleq \left\lceil \max\left\{\frac{1}{\delta_{\text{up}}}, \frac{H+1}{\widehat \delta_{\text{low}}} \right\} \right\rceil.
$$
Then, the upper-level approximate problem \eqref{eq: problem inexact optimality} is solved for $x = z^{k+1}$ with $\varepsilon_{\text{up}} \leq \delta_{\text{up}}$ and the lower-level approximate problem \eqref{eq: problem inexact feasibility natural} is solved for $x = z^{k+1}$ with $\widehat \varepsilon_{\text{low}} \leq \widehat \delta_{\text{low}}$ and the condition in step \ref{S.3} is satisfied in at most
\[
\sigma \triangleq I_{\max} \left\lceil \max\left\{ I_{\max}^8 \frac{(D+R)^4}{(1-L_\Phi)^2} C_1, I_{\max}^{\frac{8}{1-2\eta}} \frac{(D+R)^{\frac{4}{1-2\eta}}}{(1-L_\Phi)^{\frac{2}{1-2\eta}}} C_{2,\eta} \right\} \right\rceil,
\]
iterations $k$, where $\eta > 0$ is a small number, and 
\begin{equation}\label{eq: constants C1 and C2}
 C_1 \triangleq \left( D^2 + \frac{5}{4} (R+H)^2 \right)^2, \quad C_{2,\eta} \triangleq \left( \frac{(R+H)^2}{(4\eta)} \right)^{\frac{2}{1-2\eta}}.
\end{equation}
\end{theorem}
\begin{proof}
First of all we show that if $i \geq I_{\max}$, we reach the desired result. Specifically, about the upper-level problem \eqref{eq: problem inexact optimality}, we obtain
$$
\varepsilon_{\text{up}} = \bar \varepsilon^i \bar \tau^i = \frac{1}{i} \leq \delta_{\text{up}},
$$
where the first equality is due to Proposition \ref{th: upper level optimality}, and the last inequality follows from $i \geq I_{\max} \geq (\delta_{\text{up}})^{-1}$.

About the lower-level problem \eqref{eq: problem inexact feasibility natural}, preliminarily we observe that
\begin{equation}\label{eq: natural map vs dual}
\left\|P_Y(w^{i+1} - \Phi_{\bar \tau_i}(w^{i+1})) - w^{i+1}\right\|_2 \leq \sqrt{\bar \varepsilon^i},
\end{equation}
because $w^{i+1}$ satisfies the condition in step \ref{S.3} with $\bar \varepsilon^i$, see Proposition \ref{th: inexact feasibility relations}.
Moreover, we get
$$
\widehat \varepsilon_{\text{low}} \leq \frac{1}{\bar \tau^i} H + \sqrt{\bar \varepsilon^i} = \frac{H+1}{i} \leq \widehat \delta_{\text{low}},
$$
where the first inequality is due to \eqref{eq:Vuppbound} and \eqref{eq: natural map vs dual}, and the last inequality follows from $i \geq I_{\max} \geq (\widehat \delta_{\text{low}})^{-1} (H+1)$.

Now we consider the number of inner iterations needed to satisfy the condition in step \ref{S.3} with the smallest error $\bar \varepsilon^{I_{\max}} = I_{\max}^{-2}$ and for a $\tau > 0$. \textcolor{black}{Without loss of generality, in the following developments we will assume $\bar k=0$, meaning that we are simply computing the number of inner iterations.}
By \eqref{eq:intermaver 2}, the dual subproblem
\begin{equation}\label{eq: dual subproblem inexact optimality}
 \Phi_\tau (y)^\top(y-z^{k}) \geq -\varepsilon_{\text{sub\_dual}}^k, \quad \forall \, y \in Y,
\end{equation}
is solved for $\varepsilon_{\text{sub\_dual}}^k = \frac{D^2 + \sum_{j=0}^{k} (\gamma^{j})^2 (R+H)^2}{2\sum_{j=0}^{k} \gamma^{j}}$.
From Lemma \ref{th:boundsforgamma}, we obtain
$$
\sum_{j=0}^{k} \gamma^{j} \geq k^{\frac{1}{2}}, \quad \sum_{j=0}^{k} (\gamma^{j})^2 \leq \frac{5}{4} + \frac{1}{4} \ln(k) \leq \frac{5}{4} + \frac{1}{4 \eta} k^\eta, \quad \eta > 0,
$$
because
$$
\ln(k) = \int_1^k t^{-1} dt \leq \int_1^k t^{-1+\eta} dt \leq \frac{1}{\eta} k^\eta.
$$
Therefore
\begin{equation}\label{eq: dual error inner}
\begin{array}{rl}
\varepsilon_{\text{sub\_dual}}^k & \leq \frac{D^2 + \left( \frac{5}{4} + \frac{1}{4 \eta} k^\eta \right) (R+H)^2}{2 k^{\frac{1}{2}}} \\
 & = \frac{D^2 + \frac{5}{4} (R+H)^2}{2 k^{\frac{1}{2}}} + \frac{\frac{1}{4 \eta} k^\eta (R+H)^2}{2 k^{\frac{1}{2}}} \\
 & = \frac{C_1^{\frac{1}{2}}}{2 k^{\frac{1}{2}}} + \frac{C_{2,\eta}^{\frac{1-2\eta}{2}}}{2 k^{\frac{1-2\eta}{2}}} \\
 & \leq \max\left\{\frac{C_1^{\frac{1}{2}}}{k^{\frac{1}{2}}}, \frac{C_{2,\eta}^{\frac{1-2\eta}{2}}}{k^{\frac{1-2\eta}{2}}}\right\}.
\end{array}
\end{equation}
Now we show that 
\begin{equation}\label{eq: dual to natural errors}
\left\|P_Y(z^k - \Phi_\tau(z^k)) - z^k\right\|_2 \leq \frac{\sqrt{\varepsilon_{\text{sub\_dual}}^k}}{\sqrt{1-L_\Phi}}.
\end{equation}
In fact, taking $y = v^k = P_Y(z^k - \Phi_\tau(z^k)) \in Y$ in \eqref{eq: dual subproblem inexact optimality}, we have
\[
\begin{array}{rcl}
\varepsilon_{\text{sub\_dual}}^k & \ge & \Phi_\tau(v^k)^\top (z^k - v^k) \\[5pt]
& = & [z^k - v^k - [z^k - v^k - \Phi_\tau(v^k)]]^\top (z^k - v^k)\\[5pt]
& = &  \|z^k - v^k\|^2_2 - [z^k - \Phi_\tau(v^k) - v^k]^\top (z^k - v^k)\\[5pt]
& = & \|z^k - v^k\|^2_2 - [z^k - \Phi_\tau(z^k) - v^k]^\top (z^k - v^k)\\[5pt]
& & - [\Phi_\tau(z^k) - \Phi_\tau(v^k)]^\top (z^k - v^k)\\[5pt]
& \ge & (1 - L_\Phi) \|z^k - v^k\|^2_2,
\end{array}
\] 
where the last inequality follows from the Lipschitz continuity of $\Phi_\tau$ and the characteristic property of the projection. 

From Proposition \ref{th: inexact feasibility relations} and inequality \eqref{eq: dual to natural errors}, we obtain the following error for the subproblem
\begin{equation}\label{eq: primal subproblem inexact optimality}
 \Phi_\tau (z^k)^\top(y-z^{k}) \geq -\frac{D+R}{\sqrt{1-L_\Phi}}\sqrt{\varepsilon_{\text{sub\_dual}}^k}, \quad \forall \, y \in Y,
\end{equation}
and then, by \eqref{eq: dual error inner}, the desired accuracy for the subproblem is obtained when
$$
I_{\max}^{-2} = \bar \varepsilon^{I_{\max}} \geq \frac{D+R}{\sqrt{1-L_\Phi}} \max\left\{\frac{C_1^{\frac{1}{4}}}{k^{\frac{1}{4}}}, \frac{C_{2,\eta}^{\frac{1-2\eta}{4}}}{k^{\frac{1-2\eta}{4}}}\right\},
$$
that is
\begin{align*}
k \geq \max \left\{ I_{\max}^8 \frac{(D+R)^4}{(1-L_\Phi)^2} C_1, I_{\max}^\frac{8}{1-2\eta} \frac{(D+R)^\frac{4}{1-2\eta}}{(1-L_\Phi)^\frac{2}{1-2\eta}} C_{2,\eta} \right\}.
\end{align*}
The thesis follows by multiplying the number of outer iterations ($i \geq I_{\max}$) for the number of inner ones.
\hfill $\square$
\end{proof}

\noindent
In order to provide other complexity results for our method, we consider the following proposition, which is the dual counterpart of Proposition \ref{th: upper level optimality}, and provides a theoretical basis for Theorem \ref{th:ftermination easy}.

\begin{proposition}\label{th: upper level optimality dual}
Assume conditions (A1)-(A3) to hold, and let $x \in Y$ be an approximate solution of the dual VI subproblem:
\begin{equation}\label{eq: subproblem inexact optimality dual}
 \Phi_\tau (y)^\top(y-x) \geq -\varepsilon_{\text{sub\_dual}}, \quad \forall \, y \in Y,
\end{equation}
with $\tau > 0$ and $\varepsilon_{\text{sub\_dual}} \geq 0$.
It holds that $x$ turns out to be an approximate solution for the dual formulation of problem \eqref{eq:pureh}, that is
\begin{equation}\label{eq: problem inexact optimality dual}
G(y)^\top (y-x) \geq -\varepsilon_{\text{up\_dual}}, \quad \forall y \in \text{SOL}(F,Y),
\end{equation}
with $\varepsilon_{\text{up\_dual}} \geq \varepsilon_{\text{sub\_dual}} \tau$, and
\begin{equation}\label{eq: problem inexact feasibility dual}
F(y)^\top (y-x) \geq -\varepsilon_{\text{low\_dual}}, \quad \forall y \in Y,
\end{equation}
with $\varepsilon_{\text{low\_dual}} \geq \varepsilon_{\text{sub\_dual}} + \frac{1}{\tau} H D$.
\end{proposition}
\begin{proof}
 We have for all $y \in \text{SOL}(F,Y)$:
 \begin{align*}
 -\varepsilon_{\text{sub\_dual}} \tau & \leq \left[\tau F(y) + G(y)\right]^\top(y-x) \\
 & \leq G(y)^\top(y-x),
 \end{align*}
 where the first inequality is due to \eqref{eq: subproblem inexact optimality dual} and the last one is true because $x \in Y$ and then $F(y)^\top(x-y) \geq 0$. That is \eqref{eq: problem inexact optimality dual} is true.

 Moreover, we have for all $y \in Y$:
 \begin{align*}
 F(y)^\top (y-x) & = \Phi_\tau (y)^\top (y-x) - \frac{1}{\tau} G(y)^\top (y-x) \\
 & \geq -\varepsilon_{\text{sub\_dual}} - \frac{1}{\tau} H D,
 \end{align*}
 where the inequality is due to \eqref{eq: subproblem inexact optimality dual}. Therefore we get \eqref{eq: problem inexact feasibility dual}.
 \hfill $\square$
\end{proof}

\noindent
The following theorem considers a simplified version of PATA. Specifically, the parameter $\tau$ is right away initialized to a value sufficiently large to get the prescribed optimality accuracy. Moreover, approximate optimality for problem \eqref{eq:pureh} is considered only in its dual version. That said, the complexity bound obtained is better than the one given by Theorem \ref{th:ftermination}.

\begin{theorem}\label{th:ftermination easy}
Assume conditions (A1)-(A3) to hold. Consider PATA. Given some precision $\delta \in (0,1)$, set $\gamma^k = \min\left\{1,\frac{1}{2 k^\frac{1}{2}}\right\}$, $\bar \tau^0 = \bar I_{\max}$, and $\bar \varepsilon^0 = 0$ where
$$
 \bar I_{\max} \triangleq \left\lceil \frac{H+1}{\delta} \right\rceil.
$$
Then, the upper-level approximate dual problem \eqref{eq: problem inexact optimality dual} is solved for $x = z^{k+1}$ with $\varepsilon_{\text{up\_dual}} \leq \delta$ and the lower-level approximate dual problem \eqref{eq: problem inexact feasibility dual}
is solved for $x = z^{k+1}$ with $\varepsilon_{\text{low\_dual}} \leq \delta$ in at most
\[
\bar \sigma \triangleq \left\lceil \max\left\{ \bar I_{\max}^4 C_1, \bar I_{\max}^{\frac{4}{1-2\eta}} C_{2,\eta} \right\} \right\rceil,
\]
iterations $k$, where $\eta > 0$ is a small number, and $C_1$ and $C_{2,\eta}$ are given in \eqref{eq: constants C1 and C2}.
\end{theorem}
\begin{proof}
First of all we denote with $\varepsilon_{\text{sub\_dual}}^k$ the error with which the current iteration solves the dual subproblem \eqref{eq: subproblem inexact optimality dual}.
Notice that as soon as $\varepsilon_{\text{sub\_dual}}^k \leq \bar I_{\max}^{-2}$, the desired accuracy for both the upper- and the lower-level dual problems is reached. In fact, as done in the proof of Theorem \ref{th:ftermination}, and considering Proposition \ref{th: upper level optimality dual}, we have:
\begin{equation*}
\varepsilon_{\text{up\_dual}} = \varepsilon_{\text{sub\_dual}}^k \bar \tau^0 = \frac{1}{\bar I_{\max}} \leq \delta
\end{equation*}
\textcolor{black}{where the first equality is due to Proposition \ref{th: upper level optimality}, and the last inequality follows from $i \geq \bar I_{\max} \geq \delta^{-1}$, and}
\begin{equation*}
\varepsilon_{\text{low\_dual}} \leq \frac{H}{\bar \tau^0} + \sqrt{\varepsilon_{\text{sub\_dual}}^k} = \frac{H+1}{\bar I_{\max}} \leq \delta.
\end{equation*}
\textcolor{black}{where the first inequality is due to \eqref{eq:Vuppbound} and \eqref{eq: natural map vs dual}, and the last inequality follows from $i \geq \bar I_{\max} \geq \delta^{-1} (H+1)$.}

\noindent
By \eqref{eq: dual error inner} we have
$$
\varepsilon_{\text{sub\_dual}}^k \leq \max\left\{\frac{C_1^{\frac{1}{2}}}{k^{\frac{1}{2}}}, \frac{C_{2,\eta}^{\frac{1-2\eta}{2}}}{k^{\frac{1-2\eta}{2}}}\right\}.
$$
Therefore, $\varepsilon_{\text{sub\_dual}}^k \leq \bar I_{\max}^{-2}$ is implied by
$$
\bar I_{\max}^{-2} \geq \max\left\{\frac{C_1^{\frac{1}{2}}}{k^{\frac{1}{2}}}, \frac{C_{2,\eta}^{\frac{1-2\eta}{2}}}{k^{\frac{1-2\eta}{2}}}\right\},
$$
that is
\begin{align*}
k \geq \max \left\{ I_{\max}^4 C_1, I_{\max}^\frac{4}{1-2\eta} C_{2,\eta} \right\},
\end{align*}
and the thesis follows.
\hfill $\square$
\end{proof}

\section{Numerical experiments}\label{sec: numerical}

We now tackle a practical example which is representative of the fact that, under assumptions (A1)-(A3), PATA produces the sequence of points $\{z^k\}$ that is (subsequentially) convergent to a solution of the hierarchical problem, while the sequence $\{y^k\}$ never approaches the solution set. Notice that $\{y^k\}$ coincides with the sequence produced by the Tikhonov methods proposed in \cite{lampariello2020explicit} when no proximal term is considered.

Let us examine the selection problem \eqref{eq:pureh}, where:
\[
G(y)=\begin{pmatrix}0 & -\frac{1}{2}\\ \frac{1}{2} & 0\end{pmatrix} \begin{pmatrix} y_1\\ y_2\end{pmatrix}, \, F(y) = \begin{pmatrix}0 & 1\\ -1 & 0\end{pmatrix} \begin{pmatrix} y_1\\ y_2\end{pmatrix}, \, Y = \mathbb B(0,1),
    \]
where $\mathbb B(0,1)$ denotes the unit ball.    
The unique feasible point and, thus, the unique solution of the problem is $z^* = (0,0)\trt$. The assumptions (A1)-(A3) are satisfied, but notice that $G$ does not satisfy convergence conditions of the Tikhonov-like methods proposed in \cite{facchinei2014vi,lampariello2020explicit} because it is not monotone plus.

The generic $k$th iteration of PATA, in this case, should read as reported below:
\[
y^{k+1} = P_Y(y^{k} - \gamma^k [F(y^{k}) + \frac{1}{\tau^k} G(y^{k})]),
\]
where we take, for example, but without loss of generality, $\tau^k = \tau \geq 1$ and $\gamma^k = \gamma > 0$. 
We remark that the unique exact solution of the VI subproblem \eqref{eq: subproblem inexact optimality} is the origin, and then every inexact solution, with a reasonably small error, cannot be far from it.
For every $k$ it holds that
\[
y^{k+1}= P_Y\left(\begin{pmatrix}
1 & \gamma(\frac{1}{2\tau}-1)\\ \gamma(1 - \frac{1}{2\tau}) & 1	
\end{pmatrix}
\begin{pmatrix} y_1^{k}\\ y_2^{k}\end{pmatrix}\right),\\[5pt]
\] 
hence $\|y^{k+1}\|_2 = \min \left\{1, \sqrt{1+\gamma^2(\frac{1}{2\tau}-1)^2}\ \|y^k\|_2 \right\}$.
Therefore we consider $\|y^0\|_2 = 1$ and get $\|y^{k}\|_2 = 1$ for every $k$, because $\sqrt{1+\gamma^2(\frac{1}{2\tau}-1)^2}\ > 1$.
Therefore, neither does the sequence $\{y^k\}$ produced by PATA lead to the unique solution $z^*$ of problem \eqref{eq:pureh}, nor does it approach the inexact solution set of the VI subproblem.

We now consider the sequence $\{z^k\}$ produced by PATA. In order to show that this sequence leads us to the solution of the hierarchical problem, we analyze a numerical implementation of the algorithm.
Some further considerations are in order before showing the actual implemented scheme.
\begin{itemize}
\item
A general rule for the update of the variable $z^{k}$
is given by the following relation:
\[
z^{k+1}= \frac{z^k \gamma^{\text{sum},l,k}+ \gamma^{k+1-l}y^{k+1}}{\gamma^{\text{sum},l,k} + \gamma^{k+1-l}},
   \]
where
 \[
\gamma^{\text{sum},l,k} \triangleq \sum_{j=l}^k \gamma^{j-l},
  \]
which gives us the expression of $z^{k+1}$ reported in Step \ref{S.2} in PATA. This is done in order to avoid keeping trace of all $y^{j}$, $j=l,..,k$, which carries a heavy computational weight. Instead, we only need to know the current value of $z^{k}$, the sum $\gamma^{\text{sum},l,k}$, $\gamma^{k+1-l}$ and, last but not least, the current point $y^{k+1}$. This allows us to save 4 entities only, which is far more convenient.
\item
Because the feasible set $Y=\mathbb B(0,1)$ is the unit ball of radius 1, the computation of the projection steps (see Step \ref{S.2}) becomes straightforward, since it is sufficient to divide the argument by its vector norm:
\[
P_{\mathbb B(0,1)}(w) = \frac{w}{\|w\|_2}       \          \forall w:\|w\|_2  \geq 1.
  \]   
Moreover, a closed-form expression for the unique solution $u$ of the minimum problem at Step \ref{S.3}:
\[
u = \arg\min_{y \in \mathbb B(0,1) }\ [F(z^{k+1}) + \frac{1}{\tau^{k}} G(z^{k+1})]^\top (y - z^{k+1})
  \]
is achievable. On the basis that the feasible set $Y=\mathbb B(0,1)$ becomes an active constraint at the optimal solution $u$, the KKT-multiplier associated to this constraint is strictly positive. We, of course, do not know the value of the multiplier itself, but we can impose that the optimal point has Euclidean norm 1, so that it belongs to the boundary of $\mathbb B(0,1)$:
\[
u = -\frac{F(z^{k+1}) + \frac{1}{\tau^{k}} G(z^{k+1})}{\|F(z^{k+1}) + \frac{1}{\tau^{k}} G(z^{k+1})\|_2}.
 \]
\end{itemize}
We now show the implemented scheme in Algorithm \ref{al:algoexample}.
\medskip
{

\IncMargin{1em}
\begin{algorithm}[H]
\KwData{$ z^0, y^0 \in \mathbb B(0,1) $, $i \leftarrow 1, \, l \leftarrow 0$, $k^\text{max} > 0$, $\text{tol} > 0$, $\gamma^{\text{sum,l,k}} \leftarrow 0$, $a > 0$, $\alpha \in (0,1]$, $\beta > 1$;}
\medskip
\For{$k=1,\ldots, k^\text{max}$}{
\nlset{(S.1)}{$\gamma^k = \min\left\{1,\frac{a}{(k-l)^\alpha}\right\}$, $\tau^{k}=i$, $\varepsilon^{k} = \frac{1}{i^\beta}$;

\nlset{(S.2)}{$v = y^k - \gamma^{k} [F(y^k) + \frac{1}{\tau^{k}} G(y^k)])$}\;
\nlset{(S.3)}\eIf{$\|v\|_2 \geq 1$}{
   $y^{k+1} =\frac{v}{\|v\|_2}$\;
   }{
   $y^{k+1} = v$\;
   }}
\nlset{(S.4)}$z^{k+1} = \frac{z^k \gamma^{\text{sum},l,k}+ \gamma^{k}y^{k+1}}{\gamma^{\text{sum},l,k} + \gamma^{k}}$\;
\nlset{(S.5)}$u = -\frac{F(z^{k+1}) + \frac{1}{\tau^{k}} G(z^{k+1})}{\|F(z^{k+1}) + \frac{1}{\tau^{k}} G(z^{k+1})\|_2}$\;

\nlset{(S.6)}\eIf{$[F(z^{k+1}) + \frac{1}{\tau^{k}} G(z^{k+1})]^\top (u - z^{k+1}) \geq -\varepsilon^k$}{\If{$\varepsilon^k \leq \textnormal{tol}$}{\textbf{break}}$i \leftarrow i+1$\; $l \leftarrow k+1$\; $\gamma^{\text{sum},l,k+1} = 0$;}
{$\gamma^{\text{sum},l,k+1} = \gamma^{\text{sum},l,k} + \gamma^k$\;}}



\Return{$z^{k+1}$.}
\caption{\label{al:algoexample}Practical version of PATA}
\end{algorithm}}
\medskip

\noindent
As far as the steps of Algorithm \ref{al:algoexample} are concerned, (S.2) and (S.3) perform step (S.2) of PATA, while (S.5) and (S.6) fulfil step (S.4) of PATA.

We set the parameters $k^{\max} = 10^6$, $\text{tol} = 10^{-3}$, $a = \alpha = \frac{1}{2}$, $\beta = 2$.
Table \ref{tab:my-table} summarizes the results obtained by running Algorithm \ref{al:algoexample}. It is clear to see how $\|z^{k+1}\|_2$ tends to 0 as the number of iterations $k$ grows, which is what we expected, being $z^*=(0, 0)\trt$ the unique solution of the problem.  

\begin{table}
\centering
\caption{Numerical experiment: results. \label{tab:my-table}}
\begin{tabular}{cccc}
\textbf{i} & \textbf{k} & \textbf{$\varepsilon^{k}$} & \textbf{$\|z^{k+1}\|_2$} \\ \hline
1          & 1          & 1.00000           & 1.00e+00          \\ 
2          & 50         & 0.25000           & 3.28e-01          \\ 
3          & 107        & 0.11111           & 1.29e-01          \\ 
4          & 165        & 0.06250           & 6.78e-02          \\ 
5          & 223        & 0.04000           & 4.05e-02          \\ 
6          & 281        & 0.02778           & 2.57e-02          \\ 
7          & 339        & 0.02041           & 2.01e-02          \\ 
8          & 540        & 0.01562           & 1.48e-02          \\ 
9          & 740        & 0.01235           & 1.20e-02          \\ 
10         & 1166       & 0.01000           & 9.73e-03          \\ 
\vdots        & \vdots        & \vdots               & \vdots               \\ 
20         & 17691      & 0.00250           & 2.55e-03          \\ 
21         & 21952      & 0.00227           & 2.32e-03          \\ 
22         & 27084      & 0.00207           & 2.10e-03          \\ 
23         & 33167      & 0.00189           & 1.92e-03          \\ 
24         & 40281      & 0.00174           & 1.77e-03          \\ 
25         & 48506      & 0.00160           & 1.63e-03          \\ 
26         & 59199      & 0.00148           & 1.49e-03          \\ 
27         & 71242      & 0.00137           & 1.39e-03          \\ 
28         & 84715      & 0.00128           & 1.30e-03          \\ 
29         & 99699      & 0.00119           & 1.21e-03          \\ 
30         & 117950     & 0.00111           & 1.13e-03          \\ 
31         & 137950     & 0.00104           & 1.06e-03          \\ 
32         & 161698     & 0.00098           & 9.88e-04          \\ \hline
\end{tabular}
\end{table}

To further reiterate the elements of novelty that PATA displays, we hereby present some numerical 
experiments in which PATA performs better than Algorithm 1 presented in \cite{lampariello2020explicit}.
We do not intend to present a thorough numerical comparison between these solution methods, we just want to show that PATA is a fundamental solution tool when the classical Tikhonov gradient method presented in \cite{lampariello2020explicit} struggles to converge.

Again, for the sake of simplicity we consider $Y=\mathbb B(0,1)$. This time, we extend the problem to encompass $n=100$ variables and consider 
$G(x) = M_{G} x + b_G$ and $F(x) = M_F$, with 
$$
M_* = \begin{pmatrix}  & & & & & v_1^{M_*} \\ & & & & \reflectbox{$\ddots$} & \\ & & & v_{\frac{n}{2}}^{M_*} & & \\ & & -v_{\frac{n}{2}}^{M_*} & & & \\ & \reflectbox{$\ddots$} & & & & \\ -v_1^{M_*} & & & & & \end{pmatrix} + \zeta u^{M_*} (u^{M_*} + 0.01 w^{M_*})^\top, \quad * = G, F,
$$
$b_G = \zeta v^{b_G}$, and $v^{M_G}$, $u^{M_G}$, $w^{M_G}$, $v^{M_F}$, $u^{M_F}$, $w^{M_F}$, and $v^{b_G}$ are randomly generated between 0 and 1, $\zeta > 0$. We remark that when $\zeta =0$ the problem is a generalization of that in the simple example described at the beginning of this section. In our experiments, we consider the cases $\zeta = 0.1$ and $0.01$.

As far as PATA parameters are concerned, for the purpose of the implementation we set $k^{\max} = 5 \cdot 10^4$, $a = 1$, $\alpha= \frac{1}{4}$ and $\beta = 2$. As for the Tikhonov scheme proposed in \cite{lampariello2020explicit}, we set $\lambda = 0.1$.

Following Proposition \ref{th: upper level optimality}, a merit function for the nested variational inequality \eqref{eq:pureh} can be given by
$$
\text{optimality measure}(k) \triangleq \max\left\{ \varepsilon_{\text{sub}}^k \tau^k, \varepsilon_{\text{sub}}^k + \frac{1}{\tau^k} \right\}.
$$
We generated 3 different instances of the problem and considered 2 values for $\zeta$, for a total of 6 different test problems.

Figure \ref{fig:numerical} shows the evolution of the optimality measure as the number of inner iterations $k$ grows towards $k^{\max}$, when both PATA and the classical Tikhonov algorithm described in \cite{lampariello2020explicit} are applied to the 6 test problems. 

It is clear to see how the practical implementation for PATA always outperforms the classical Tikhonov in \cite{lampariello2020explicit}, as it needs a significant smaller number of inner iterations $k$ to reach small values of the optimality measure. 
However, we remark that computing averages, such as in step \ref{S.2} of PATA, can be computationally expensive.
Hence, PATA becomes an essential alternative tool when other Tikhonov-like methods, not including averaging steps, either fail to reach small values of the optimality measure, as shown in our practical implementation (see Figure \ref{fig:numerical}), or do not converge at all.   

\begin{figure}
\begin{subfigure}{.5\textwidth}
  \centering
  \includegraphics[width=.8\linewidth]{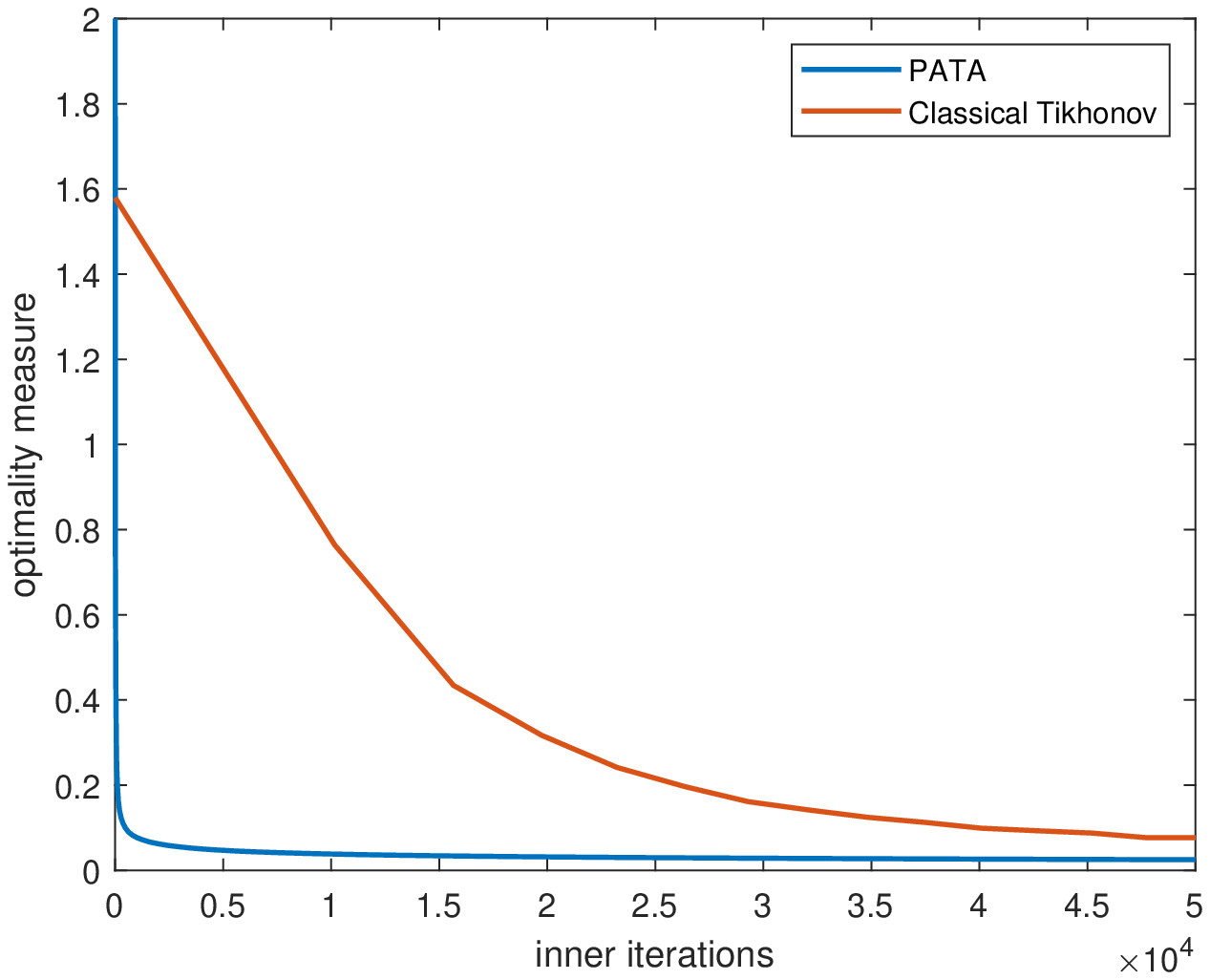}
  \caption{1a \label{fig:sfig1}}
\end{subfigure}%
\vspace{5pt}
\begin{subfigure}{.5\textwidth}
  \centering
  \includegraphics[width=.8\linewidth]{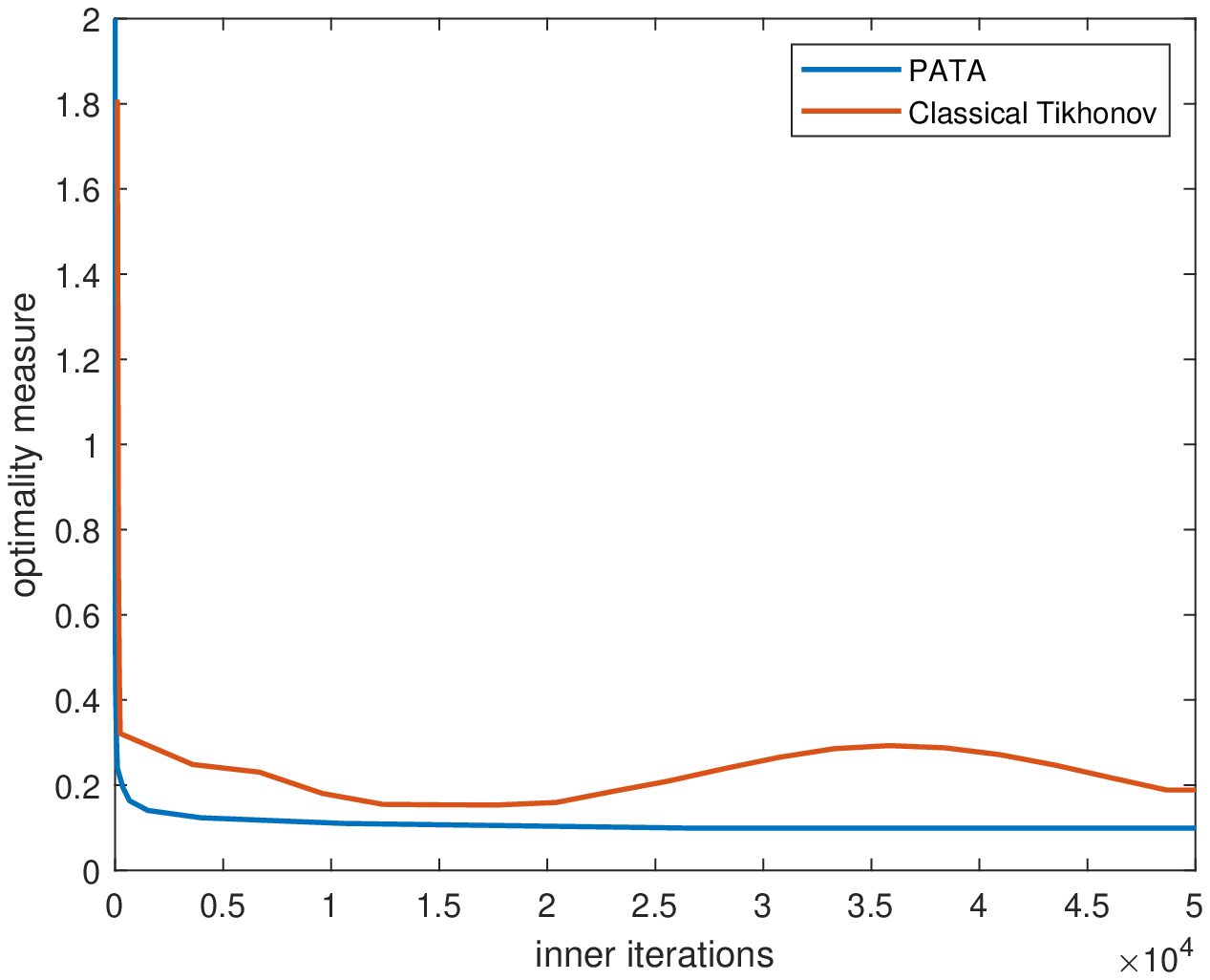}
  \caption{1b \label{fig:sfig2}}
\end{subfigure}
\vspace{5pt}
\begin{subfigure}{.5\textwidth}
  \centering
  \includegraphics[width=.8\linewidth]{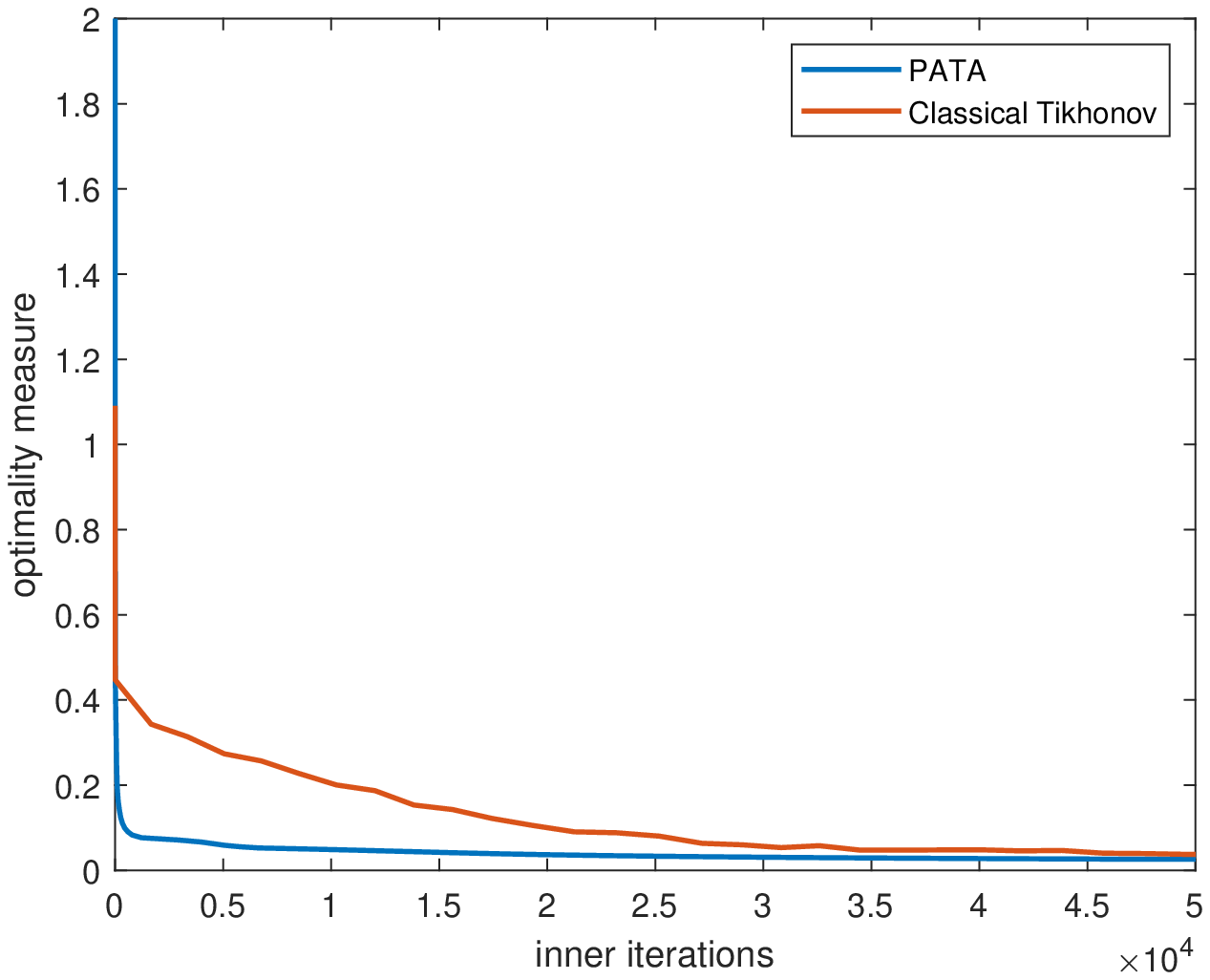}
  \caption{1c \label{fig:sfig3}}
\end{subfigure}%
\begin{subfigure}{.5\textwidth}
  \centering
  \includegraphics[width=.8\linewidth]{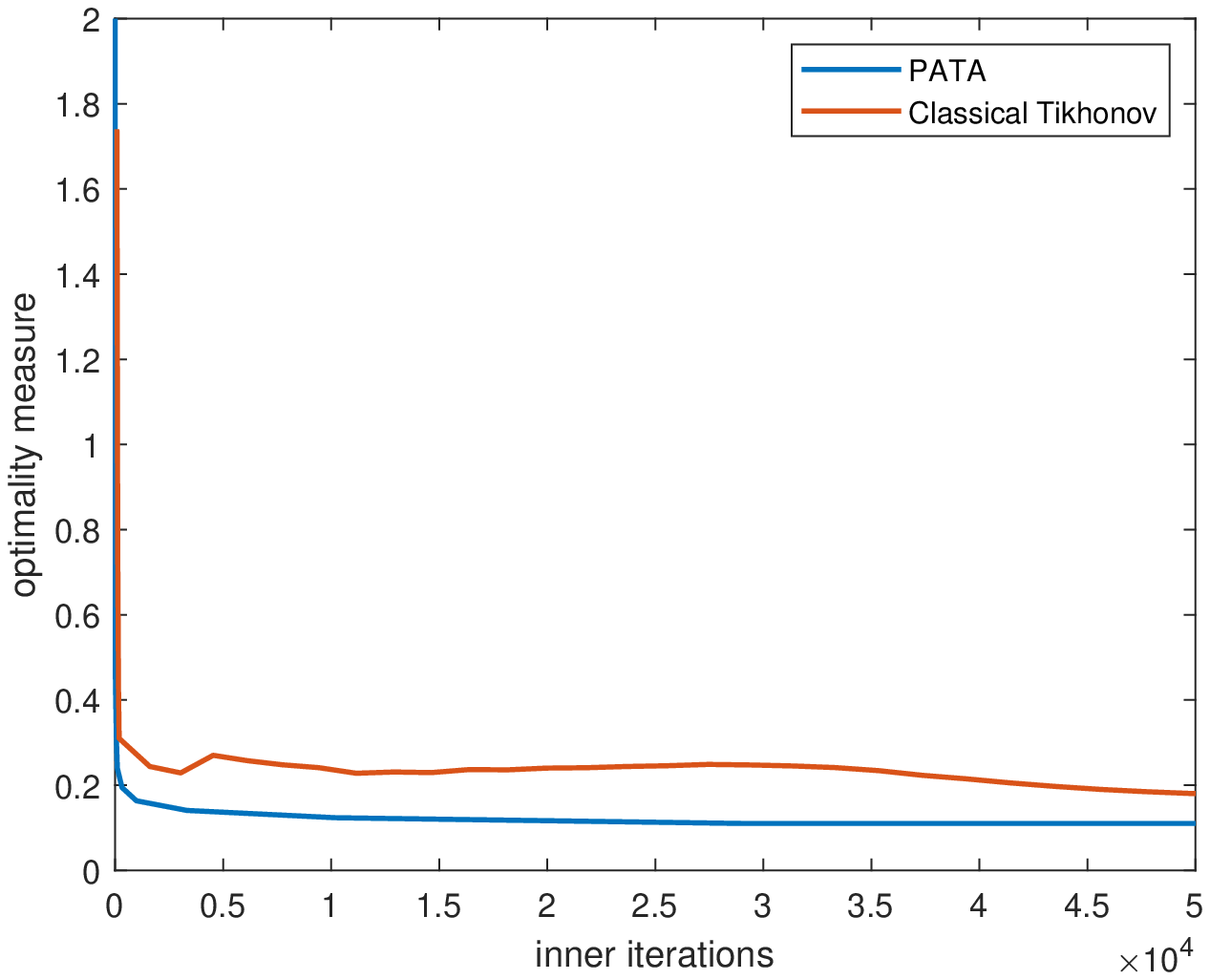}
  \caption{1d \label{fig:sfig4}}
\end{subfigure}
\begin{subfigure}{.5\textwidth}
  \centering
  \includegraphics[width=.8\linewidth]{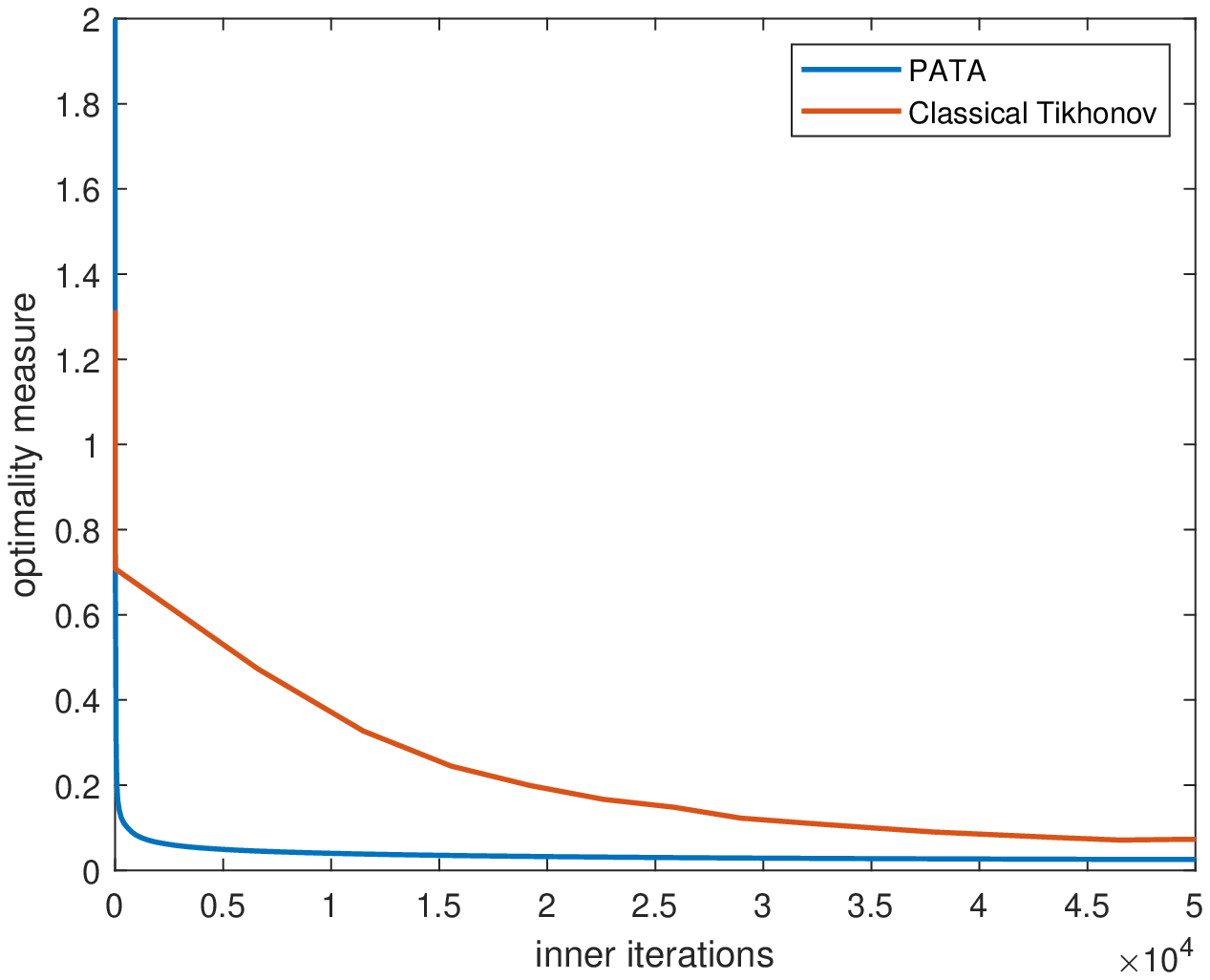}
  \caption{1e \label{fig:sfig5}}
\end{subfigure}
\begin{subfigure}{.5\textwidth}
  \centering
  \includegraphics[width=.8\linewidth]{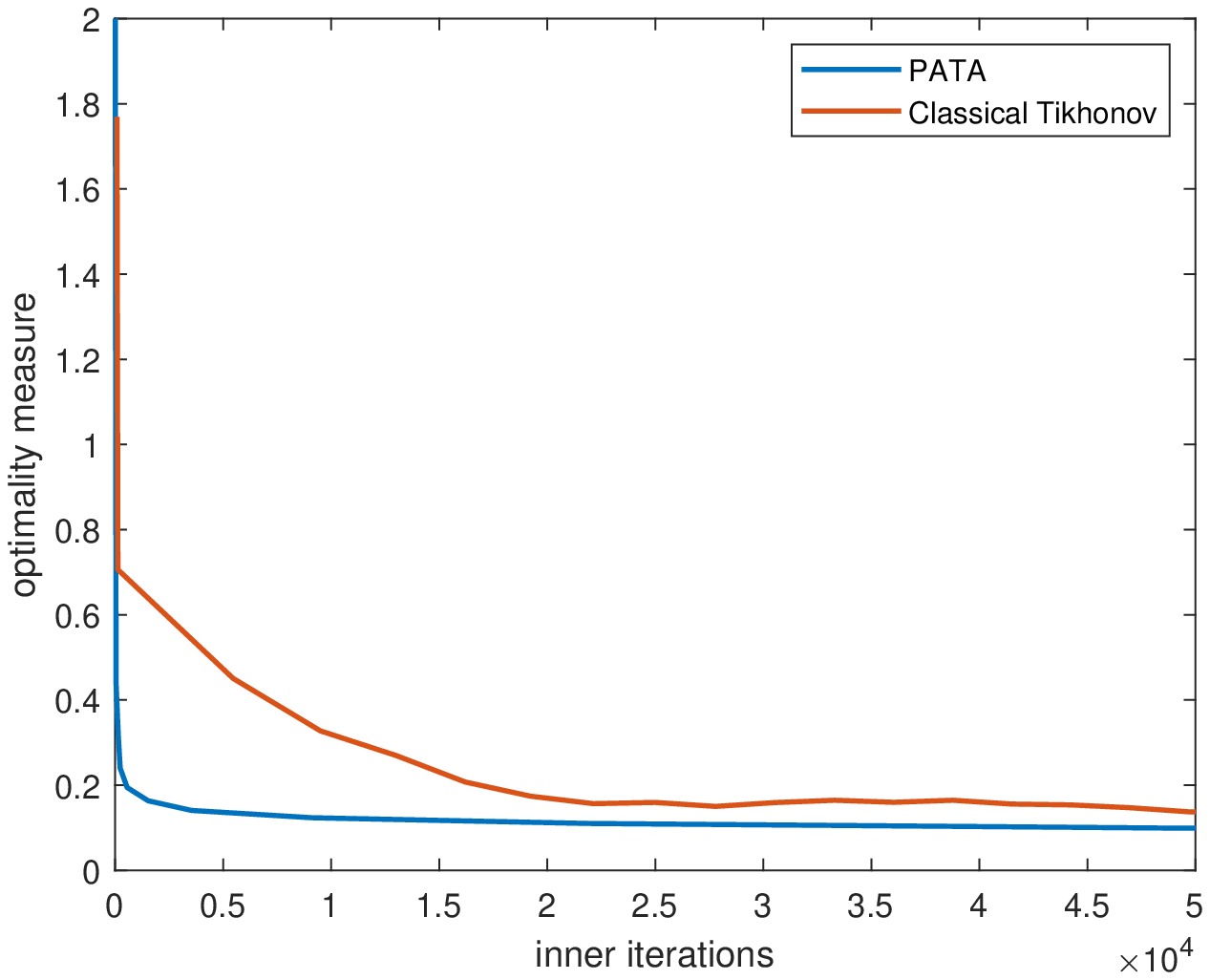}
  \caption{1f \label{fig:sfig6}}
\end{subfigure}
\caption{Plots \ref{fig:sfig1}, \ref{fig:sfig3} and \ref{fig:sfig5} correspond to the value $\zeta=0.01$; plots \ref{fig:sfig2}, \ref{fig:sfig4} and \ref{fig:sfig6} correspond to $\zeta=0.1$. Each row is related to a different instance of the problem, namely a different seed for the random generation.\label{fig:numerical}}
\end{figure}

\section{Conclusions}

We have shown that PATA is (subsequentially) convergent to solutions of monotone nested variational inequalities under the weakest conditions in the literature so far, see Theorem \ref{th:conv2}. Specifically, besides the standard convexity and monotonicity assumptions, $G$ is required to be just monotone, while all other papers demand the monotonicity plus of $G$, see \cite{facchinei2014vi,lampariello2020explicit}. 

In addition, PATA enjoys interesting complexity properties, see Theorems \ref{th:ftermination} and \ref{th:ftermination easy}. 
Notice that we have provided the first complexity analysis for nested variational inequalities considering optimality of both the upper- and lower-level. Conversely, authors in \cite{lampariello2020explicit} only handled lower-level optimality.

Possible future research may focus on generalizing the problem to consider quasi variational inequalities as well as generalized variational inequalities. The first step would be extending Proposition \ref{th: upper level optimality} to encompass these more complex variational problems. We leave this investigation to following works.

\section*{Appendix}

The following proposition is instrumental for the discussion regarding the natural map $V$ in section \ref{sec:conv}.

\begin{proposition}\label{th: inexact feasibility relations}
 Let $x \in K$ satisfy the primal VI approximate optimality condition \eqref{eq: inexact optimality}, then the natural map approximate optimality condition \eqref{eq: problem inexact feasibility natural gen} holds with $\widehat \varepsilon \geq \sqrt{\varepsilon}$.
 Vice versa, let $x \in K$ satisfy condition \eqref{eq: problem inexact feasibility natural gen}, then \eqref{eq: inexact optimality} holds with $\varepsilon \geq (\Omega + \Xi) \widehat \varepsilon$, where $\Omega \triangleq \max_{v, y \in K}\|v - y\|_2$ and $\Xi \triangleq \max_{y \in K} \|\Psi(y)\|_2$.
\end{proposition}
\begin{proof}
 Let $z = P_K(x-\Psi(x))$. If \eqref{eq: inexact optimality} holds, then
 $$
 \begin{array}{rcl}
- \varepsilon & \le & \Psi(x)^\top (z - x) =  [x - z - (x - z - \Psi(x))]^\top (z - x)\\[5pt]
& = &  - \|x - z\|_2^2 - (x - \Psi(x) - z)^\top (z - x) \leq - \|x - z\|_2^2 = - U(x)^2,
\end{array}
 $$
 where the last inequality is due to the characteristic property of the projection.
 Therefore, \eqref{eq: problem inexact feasibility natural gen} holds with $\widehat \varepsilon \geq \sqrt{\varepsilon}$.
 
 Now we consider the case in which \eqref{eq: problem inexact feasibility natural gen} holds.
 Thanks again to the characteristic property of the projection, we have
\[
[z - (x - \Psi(x))]^\top (y - z) \ge 0, \quad \forall y \in K,
\]
and, thus, for all $y \in K$,
\begin{align*}
\Psi(x)^\top (y - x) & \ge (x - z)^\top (y - z) + \Psi(x)^\top (z - x)  \\
                  & \ge - (\|y - z\|_2 + \|\Psi(x)\|_2) \|x - z\|_2 \\
                  & \ge - (\Omega + \Xi) U(x) \ge - (\Omega + \Xi) \widehat \varepsilon. 	
\end{align*}
Therefore, \eqref{eq: inexact optimality} holds with $\varepsilon \geq (\Omega + \Xi) \widehat \varepsilon$.
\hfill $\square$
\end{proof}

\noindent
The following lemma is helpful to prove Theorems \ref{th:ftermination} and \ref{th:ftermination easy}.

\begin{lemma}\label{th:boundsforgamma}
Let $\{\gamma^k\ : \gamma^k=\min\{1,\frac{a}{k^\alpha}\}\}$, \textcolor{black}{with $\alpha \in (0,1]$ and $a>0$}. Setting $K \in \mathbb{N}$, the following upper and lower bounds hold true:
\begin{enumerate}[(i)]
 \item
$\alpha \ne 1$ : $\sum_{k=0}^K \gamma^k \geq \lceil a^\frac{1}{\alpha} \rceil + \frac{a}{1-\alpha}[(K+1)^{1-\alpha} - \lceil a^\frac{1}{\alpha} \rceil^{1-\alpha}]$;  
 \item
$\alpha = 1$ : $\sum_{k=0}^K \gamma^k \geq \lceil a \rceil + \ln \left((\frac{K+1}{\lceil a \rceil})^a\right)$; 
 \item
\textcolor{black}{$\alpha \neq \frac{1}{2}$ : $\sum_{k=0}^K (\gamma^k)^2 \leq \lceil a^\frac{1}{\alpha} \rceil + \frac{a^2}{\lceil a^\frac{1}{\alpha} \rceil ^{2\alpha}} + \frac{a^2}{1-2\alpha}[K^{1-2\alpha} - \lceil a^\frac{1}{\alpha} \rceil^{1-2\alpha}]; $}
 \item
$\alpha = \frac{1}{2}$ : $\sum_{k=0}^K (\gamma^k)^2 \leq \lceil a^2 \rceil + \frac{a^2}{\lceil a^2 \rceil}+ a^2\ln(\frac{K}{\lceil a^2 \rceil}).$
 \end{enumerate} 	
\end{lemma}
\begin{proof}
In cases $(i)$ and $(ii)$:
\begin{equation*}\label{eq:chain}
\begin{array}{rcl}
\sum_{k=0}^K \gamma^k & = & \sum_{k=0}^K\min\{1,\frac{a}{k^\alpha}\} \\[5pt]
& = & \lceil a^\frac{1}{\alpha} \rceil + \sum_{k=\lceil a^\frac{1}{\alpha} \rceil}^K \frac{a}{k^\alpha}\\[5pt]
& \geq & \lceil a^\frac{1}{\alpha} \rceil + \int_{ \lceil a^\frac{1}{\alpha} \rceil }^{K+1} ax^{-\alpha} dx,
\end{array}
\end{equation*}
where the inequality is due to the integral test for Harmonic series. When $\alpha \ne 1$ it follows that:
\begin{equation*}\label{eq:chain}
\begin{array}{rcl}
\sum_{k=0}^K \gamma^k & \geq & \lceil a^\frac{1}{\alpha} \rceil + a \frac{x^{1-\alpha}}{1-\alpha}\big|_{ \lceil a^\frac{1}{\alpha} \rceil }^{K+1}\\[5pt]
& = & \lceil a^\frac{1}{\alpha} \rceil + \frac{a}{1-\alpha}[(K+1)^{1-\alpha} - \lceil a^\frac{1}{\alpha} \rceil^{1-\alpha}],
\end{array}
\end{equation*}
whilst, if $\alpha = 1$, it follows that:
\begin{equation*}\label{eq:chain}
\begin{array}{rcl}
\sum_{k=0}^K \gamma^k & \geq & \lceil a \rceil + a \ln\lvert x\rvert\big|_{ \lceil a \rceil }^{K+1}\\[5pt]
& = & \lceil a \rceil + \ln \left( (\frac{K+1}{\lceil a \rceil})^a \right).
\end{array}
\end{equation*}
\textcolor{black}{In cases $(iii)$ and $(iv)$: 
\begin{equation*}\label{eq:chain}
\begin{array}{rcl}
\sum_{k=0}^K (\gamma^k)^2 & = & \sum_{k=0}^K(\min\{1,\frac{a}{k^\alpha}\})^2 \\[5pt] 
& = & \lceil a^\frac{1}{\alpha} \rceil + \sum_{k=\lceil a^\frac{1}{\alpha} \rceil}^K \frac{a^2}{k^{2\alpha}}\\[5pt]
& = & \lceil a^\frac{1}{\alpha} \rceil + \frac{a^2}{\lceil a^\frac{1}{\alpha} \rceil ^{2\alpha}} + \sum_{k=\lceil a^\frac{1}{\alpha} \rceil + 1}^K \frac{a^2}{k^{2\alpha}} \\[5pt]
& \leq & \lceil a^\frac{1}{\alpha} \rceil + \frac{a^2}{\lceil a^\frac{1}{\alpha} \rceil ^{2\alpha}} + \int_{ \lceil a^\frac{1}{\alpha} \rceil}^{K} a^2x^{-2\alpha} dx,
\end{array}
\end{equation*}
where, in the third equality, the first term of the series is taken out, whilst the inequality is once again due to the integral test for Harmonic series.}

\textcolor{black}{When $\alpha \neq \frac{1}{2}$ it follows that:
\begin{equation*}\label{eq:chain}
\begin{array}{rcl}
\sum_{k=0}^K (\gamma^k)^2 & \leq & \lceil a^\frac{1}{\alpha} \rceil + \frac{a^2}{\lceil a^\frac{1}{\alpha} \rceil ^{2\alpha}} + a^2 \frac{x^{1-2\alpha}}{1-2\alpha}\big|_{ \lceil a^\frac{1}{\alpha} \rceil }^{K}\\[5pt]
& = & \lceil a^\frac{1}{\alpha} \rceil + \frac{a^2}{\lceil a^\frac{1}{\alpha} \rceil ^{2\alpha}} + \frac{a^2}{1-2\alpha}[K^{1-2\alpha} - \lceil a^\frac{1}{\alpha} \rceil^{1-2\alpha}],
\end{array}
\end{equation*}    }
while, if $\alpha = \frac{1}{2}$, it follows that:
\begin{equation*}\label{eq:chain}
\begin{array}{rcl}
\sum_{k=0}^K (\gamma^k)^2 & \leq & \lceil a^2 \rceil + \frac{a^2}{\lceil a^2 \rceil} + a^2\ln\lvert x\rvert\big|_{ \lceil a^2 \rceil }^{K}\\[5pt]
& = & \lceil a^2 \rceil + \frac{a^2}{\lceil a^2 \rceil}+ a^2\ln(\frac{K}{\lceil a^2 \rceil}).  
\end{array}
\end{equation*}
\hfill $\square$
\end{proof}

\noindent
The following proposition enables us to use $\gamma^k=\min\{1,\frac{a}{k^\alpha}\}$, which is shown to satisfy conditions \eqref{eq: conditions for convergence} in Theorem \ref{th:conv2}.

\begin{proposition}\label{th: sequence complies with theorem}
The sequence of stepsizes $\{\gamma^k\ : \gamma^k=\min\{1,\frac{a}{k^\alpha}\}\}$, with $\alpha \in (0,1]$ and $a>0$, satisfies both sets of hypotheses:  
\begin{enumerate}[(i)]
\item
$\sum_{k=0}^\infty \gamma^k = \infty $
\item
$\frac{\sum_{k=0}^\infty (\gamma^k)^2}{\sum_{k=0}^\infty \gamma^k} = 0$
\end{enumerate}
needed for Theorem \ref{th:conv2} to be valid, see conditions \eqref{eq: conditions for convergence}.
\end{proposition} 
\begin{proof}
$(i)$ We need to examine two separate cases, i.e. when $\alpha=1$ and $\alpha<1$, when K goes to $+\infty$.

When $\alpha=1$:
 \begin{align*}
 \sum_{k=0}^\infty \gamma^k & \geq \lim_{K \to +\infty} \left( \lceil a \rceil + \ln \left( \left(\frac{K+1}{\lceil a \rceil}\right)^a \right) \right) = +\infty.
 \end{align*}
When, instead, $\alpha<1$:
 \begin{align*}
 \sum_{k=0}^\infty \gamma^k & \geq \lim_{K \to +\infty} \left( \lceil a^\frac{1}{\alpha} \rceil + \frac{a}{1-\alpha}[(K+1)^{1-\alpha} - \lceil a^\frac{1}{\alpha} \rceil^{1-\alpha}] \right) = +\infty.
 \end{align*}
Notice that, if $\alpha>1$, the limit would be $\lceil a^\frac{1}{\alpha} \rceil - \frac{a}{1-\alpha} \lceil a^\frac{1}{\alpha} \rceil^{1-\alpha} \ne +\infty$, hence why it is necessary that $\alpha \leq 1$. \\ \\
$(ii)$ We first show that, if $\alpha \in \left( \frac{1}{2},1 \right]$, the following relation holds:
\begin{align*}
 \sum_{k=0}^\infty (\gamma^k)^2 & \leq \lim_{K \to +\infty} \textcolor{black}{\left( \lceil a^\frac{1}{\alpha} \rceil + \frac{a^2}{\lceil a^\frac{1}{\alpha} \rceil ^{2\alpha}} + \frac{a^2}{1-2\alpha}[K^{1-2\alpha} - \lceil a^\frac{1}{\alpha} \rceil^{1-2\alpha}] \right)  }\\
 & = \lceil a^\frac{1}{\alpha} \rceil + \frac{a^2}{\lceil a^\frac{1}{\alpha} \rceil ^{2\alpha}} - \frac{a^2}{1-2\alpha} \lceil a^\frac{1}{\alpha} \rceil ^{1-2\alpha} \\
 & = C,
 \end{align*}
which, in turn, implies that:
\begin{align*}
 \frac{\sum_{k=0}^\infty (\gamma^k)^2}{\sum_{k=0}^\infty \gamma^k} & \leq \frac{C}{+\infty} = 0.
 \end{align*}
When $\alpha=\frac{1}{2}$, it easy to see that:
\begin{align*}
 \frac{\sum_{k=0}^\infty (\gamma^k)^2}{\sum_{k=0}^\infty \gamma^k} & \leq \lim_{K \to +\infty} \frac{\lceil a^2 \rceil + \frac{a^2}{\lceil a^2 \rceil}+ a^2\ln \left(\frac{K}{\lceil a^2 \rceil} \right)}{\lceil a^2 \rceil + 2a[\sqrt{K+1} - \lceil a^2 \rceil^{\frac{1}{2}}]}   \\
 & = \lim_{K \to +\infty} \frac{\left( \frac{a^2 \lceil a^2 \rceil }{K} \right)}{\left( \frac{2a}{2\sqrt{K+1}} \right)} \\
 & = \lim_{K \to +\infty} \frac{a(\lceil a^2 \rceil)\sqrt{K+1}}{K} \\ 
 & = \lim_{K \to +\infty} a(\lceil a^2 \rceil)\sqrt{\frac{1}{K}+\frac{1}{K^2}} \\
 & = 0,
 \end{align*}
where the first equality follows from L'Hopital's rule.

Finally, if $\alpha \in \left(0,\frac{1}{2}\right)$, it is once again easy to see that:
 \begin{align*}
 \frac{\sum_{k=0}^\infty (\gamma^k)^2}{\sum_{k=0}^\infty \gamma^k} & \leq \lim_{K \to +\infty} \frac{\lceil a^\frac{1}{\alpha} \rceil + \frac{a^2}{\lceil a^\frac{1}{\alpha} \rceil ^{2\alpha}} + \frac{a^2}{1-2\alpha}[K^{1-2\alpha} - \lceil a^\frac{1}{\alpha} \rceil^{1-2\alpha}]}{\lceil a^\frac{1}{\alpha} \rceil + \frac{a}{1-\alpha}[(K+1)^{1-\alpha} - \lceil a^\frac{1}{\alpha} \rceil^{1-\alpha}]}  \\
 & = \lim_{K \to +\infty} \frac{\frac{a^2}{1-2\alpha}(1-2\alpha)K^{-2\alpha}}{\frac{a}{1-\alpha}(1-\alpha)(K+1)^{-\alpha}}\\
 & = \lim_{K \to +\infty} \frac{aK^{-2\alpha}}{(K+1)^{-\alpha}} \\
 & = \lim_{K \to +\infty} \frac{aK^{-\alpha}}{(1+\frac{1}{K})^{-\alpha}}  \\
 & = 0,
 \end{align*}
where the first equality follows again from L'Hopital's rule, and the last equality is true because $\alpha > 0$. \hfill $\square$ \\
\end{proof}

\noindent
This last lemma is again used in the proof of Theorem \ref{th:ftermination}.

\begin{lemma}\label{th:stopmeasure}
The following upper bound holds for the lower-level merit function $V$ (see the definition \eqref{eq:meritV}) at $z$ for every positive $\tau$: 
\begin{equation}\label{eq:Vuppbound}
V(z) \le \frac{1}{\tau} \|G(z)\|_2 + \left\|P_Y(z - \Phi_\tau(z)) - z\right\|_2.  	
\end{equation}
\end{lemma}
\begin{proof}
The claim is a consequence of the following chain of relations:
\begin{equation*}\label{eq:chain}
\begin{array}{rcl}
V(z) & = & \|P_Y(z - F(z)) - z\|_2\\[5pt]
& \le & \|P_Y(z - F(z)) - P_Y(z - (F(z) + \frac{1}{\tau} G(z))\|_2\\[5pt]
& & + \|P_Y(z - (F(z) + \frac{1}{\tau} G(z)) - z\|_2\\[5pt]
& \le & \frac{1}{\tau} \|G(z)\|_2 + \|P_Y(z - (F(z) + \frac{1}{\tau} G(z)) - z\|_2,\\[5pt]
\end{array}
\end{equation*}
where the last inequality follows from the nonexpansive property of the projection mapping. \hfill $\square$
\end{proof}

%
%

\bibliographystyle{spmpsci}      
\bibliography{Surbibtikh}   

\begin{thebibliography}{10}
\providecommand{\url}[1]{{#1}}
\providecommand{\urlprefix}{URL }
\expandafter\ifx\csname urlstyle\endcsname\relax
  \providecommand{\doi}[1]{DOI~\discretionary{}{}{}#1}\else
  \providecommand{\doi}{DOI~\discretionary{}{}{}\begingroup
  \urlstyle{rm}\Url}\fi

\bibitem{bigi2021combining}
Bigi, G., Lampariello, L., Sagratella, S.: Combining approximation and exact
  penalty in hierarchical programming.
\newblock Optimization pp. 1--17 (2021)

\bibitem{dempe2002foundations}
Dempe, S.: Foundations of bilevel programming.
\newblock Springer Science \& Business Media (2002)

\bibitem{FacchPangBk}
Facchinei, F., Pang, J.S.: Finite-{D}imensional {V}ariational {I}nequalities
  and {C}omplementarity {P}roblems.
\newblock Springer (2003)

\bibitem{facchinei2014vi}
Facchinei, F., Pang, J.S., Scutari, G., Lampariello, L.: {VI}-constrained
  hemivariational inequalities: distributed algorithms and power control in
  ad-hoc networks.
\newblock Mathematical Programming \textbf{145}(1-2), 59--96 (2014)

\bibitem{facchinei2015parallel}
Facchinei, F., Scutari, G., Sagratella, S.: Parallel selective algorithms for
  nonconvex big data optimization.
\newblock IEEE Trans. on Signal Processing \textbf{63}(7), 1874--1889 (2015)

\bibitem{kalashnikov1996solving}
Kalashnikov, V.V., Kalashinikova, N.I.: Solving two-level variational
  inequality.
\newblock Journal of Global Optimization \textbf{8}(3), 289--294 (1996)

\bibitem{lampariello2020explicit}
Lampariello, L., Neumann, C., Ricci, J.M., Sagratella, S., Stein, O.: An
  explicit {T}ikhonov algorithm for nested variational inequalities.
\newblock Computational Optimization and Applications \textbf{77}(2), 335--350
  (2020)

\bibitem{lampariello2021equilibrium}
Lampariello, L., Neumann, C., Ricci, J.M., Sagratella, S., Stein, O.:
  Equilibrium selection for multi-portfolio optimization.
\newblock European Journal of Operational Research  (2021)

\bibitem{lampariello2017bridge}
Lampariello, L., Sagratella, S.: A bridge between bilevel programs and {N}ash
  games.
\newblock Journal of Optimization Theory and Applications \textbf{174}(2),
  613--635 (2017)

\bibitem{lampariello2020numerically}
Lampariello, L., Sagratella, S.: Numerically tractable optimistic bilevel
  problems.
\newblock Comput. Optim. Appl. \textbf{76}(2), 277--303 (2020)

\bibitem{lampariello2019standard}
Lampariello, L., Sagratella, S., Stein, O.: The standard pessimistic bilevel
  problem.
\newblock SIAM Journal on Optimization \textbf{29}(2), 1634--1656 (2019)

\bibitem{lu2009hybrid}
Lu, X., Xu, H.K., Yin, X.: Hybrid methods for a class of monotone variational
  inequalities.
\newblock Nonlinear Analysis: Theory, Methods \& Applications \textbf{71}(3-4),
  1032--1041 (2009)

\bibitem{marino2011explicit}
Marino, G., Xu, H.K.: Explicit hierarchical fixed point approach to variational
  inequalities.
\newblock Journal of Optimization Theory and Applications \textbf{149}(1),
  61--78 (2011)

\bibitem{scutari2012equilibrium}
Scutari, G., Facchinei, F., Pang, J.S., Lampariello, L.: Equilibrium selection
  in power control games on the interference channel.
\newblock In: 2012 Proceedings IEEE INFOCOM, pp. 675--683. IEEE (2012)

\bibitem{yamada2001hybrid}
Yamada, I.: The hybrid steepest descent method for the variational inequality
  problem over the intersection of fixed point sets of nonexpansive mappings.
\newblock Inherently parallel algorithms in feasibility and optimization and
  their applications \textbf{8}, 473--504 (2001)

\end{thebibliography}

\end{document}